\pdfoutput=1
\documentclass[11pt]{article}

\usepackage[a4paper,margin=28mm]{geometry}
\usepackage[T1]{fontenc}
\usepackage[utf8]{inputenc}
\usepackage{lmodern}
\usepackage{amsmath,amssymb,amsthm,mathtools}
\mathtoolsset{showonlyrefs}
\usepackage{bm}
\usepackage{mathrsfs}
\usepackage{comment}

\usepackage{graphicx}
\usepackage{booktabs}
\usepackage{array}
\usepackage{enumitem}

\usepackage[colorlinks=true,linkcolor=blue,citecolor=blue,urlcolor=blue]{hyperref}
\usepackage[nameinlink,capitalize]{cleveref}

\theoremstyle{plain}
\newtheorem{theorem}{Theorem}[section]
\newtheorem{lemma}[theorem]{Lemma}
\newtheorem{proposition}[theorem]{Proposition}
\newtheorem{corollary}[theorem]{Corollary}

\theoremstyle{definition}

\newtheorem{assumption}[theorem]{Assumption}

\theoremstyle{remark}
\newtheorem{remark}[theorem]{Remark}

\DeclareMathOperator{\diag}{diag}

\title{Exact-curved Lagrange finite elements for the Poisson problem in two dimensions}

\author{
  Hiroki Ishizaka\\
  Team FEM, Matsuyama, Japan\\
  E-mail: \texttt{h.ishizaka005@gmail.com}
}

\date{}

\begin{document}

\maketitle

\begin{abstract}
Conforming mapped Lagrange finite elements of any fixed degree are analysed for the Poisson problem on exact curved triangulations in two dimensions.  The element map is factorised into an anisotropic affine core map and a smooth curved correction.  Assuming compatible edge parametrisations, semi-regular affine cores and uniform bounds on the curved corrections, we obtain anisotropic $L^2$- and $H^1$-interpolation estimates with constants independent of the aspect ratios of the cores.  The estimates retain lower-order terms induced by the non-affine maps, and these terms cannot be omitted in general.  For sufficiently regular solutions, the Galerkin errors on the exact domain converge with the optimal orders in the energy norm and, on convex domains, in the $L^2$ norm.  A fixed smooth diffeomorphism maps a semi-regular family of polygonal triangulations to exact curved triangulations satisfying these assumptions, even when the core aspect ratios are unbounded. An explicit family with an exact circular boundary arc is given.  Numerical experiments with linear and quadratic elements confirm the predicted rates. At a fixed horizontal mesh width, the errors remain bounded as the aspect ratio increases.
\end{abstract}



\section{Introduction} \label{intro}
Finite element methods on curved domains have been studied in several settings. One line of work treats interpolation and approximation directly on curved elements; see Ciarlet and Raviart \cite{CiaRav72}, Zl{\'a}mal \cite{Zla73}, Scott \cite{Sco73}, Lenoir \cite{Len86}, and Bernardi \cite{Ber89}. Bernardi's interpolation theory is particularly relevant here. In \cite{Ber89}, curved finite elements that fit the physical boundary exactly are considered, and optimal interpolation estimates are derived under regularity conditions on the family of curved elements. There, a curved simplex is written as a perturbation of an affine simplex,
\begin{align*}
G_K = \widetilde G_K + \phi_K,
\end{align*}
where $\widetilde G_K$ is an invertible affine map and $\phi_K$ is a sufficiently small curved perturbation.

Another approach is based on domain perturbation. For example, Elliott and Ranner \cite{EllRan13} analysed finite element methods for coupled bulk--surface elliptic problems by posing the discrete problem on an approximate domain and comparing it with the continuous problem on the physical domain through lift operators. In this approach, geometric approximation is treated as a variational crime and estimated separately from the finite element approximation error. For finite element methods for surface PDEs, we refer to Dziuk and Elliott \cite{DziEll13}.

Geometric transformation provides another way of treating curved domains. Holst and Licht \cite{HolLic23} transform a Poisson problem to a polytopal parametric domain and use a broken Bramble--Hilbert argument to recover higher-order convergence rates. Their polynomial approximation constants are allowed to depend on a shape measure of the parametric triangulation. Kawecki \cite{Kaw20} develops interpolation, trace, inverse and curvature estimates for curved simplices, with applications to discontinuous Galerkin methods. Our concern is how the interpolation bounds depend on strongly anisotropic element geometry, rather than the use of curved maps or geometric transformations itself.

The present paper follows the exact-domain approach of Bernardi \cite{Ber89}. In both the analysis and the numerical experiments, the images of the element maps cover the physical domain exactly; the quadrature error is checked separately. Unlike the classical curved-element interpolation theory, we write the element map as
\begin{align*}
F = \Psi \circ \Phi,
\end{align*}
where $\Phi$ is the affine map defined in Section~\ref{ref=affine=subsec} and $\Psi$ is the curved correction defined in Section~\ref{subsec:exact-curved-triangles}. This factorisation separates the anisotropic affine scaling from the non-affine contribution of the curved correction.

Although Bernardi's map $G_K = \widetilde G_K + \phi_K$ can formally be rewritten as a composition
\begin{align*}
G_K
=
\left(\operatorname{id}+\phi_K \circ \widetilde G_K^{-1}\right) \circ \widetilde G_K,
\end{align*}
the analysis in \cite{Ber89} does not separate the affine scaling from the curvature. Here, this separation allows us to transfer the anisotropic interpolation estimates of \cite{IshKobTsu23} from the affine cores to exact curved elements, while keeping the dependence on the curved correction explicit. The classical curved-element estimates \cite{CiaRav72,Zla73,Len86,Ber89} are commonly stated for regular families of curved elements. In the present analysis, the semi-regularity parameter of the affine core and the derivatives of the curved correction are controlled separately. This gives directional interpolation estimates on exact curved elements with constants independent of the aspect ratios of the affine cores, even when those ratios tend to infinity.

Our main results are as follows. We transfer the affine anisotropic interpolation estimates to exact curved elements for any fixed Lagrange degree $k$, with constants depending on the semi-regularity parameter and the bounds for the curved correction, but not on the aspect ratios of the affine cores (Theorem~\ref{thm:high-order-curved}). The estimates contain lower-order terms generated by the non-affine maps, and Remark~\ref{rem:full-norm} shows that these terms cannot be dropped. We show that any semi-regular triangulation of a polygon, mapped by a fixed smooth diffeomorphism, satisfies the geometric assumptions, including the compatibility of the edge parametrisations (Proposition~\ref{prop:global-map}); an explicit example has an exact circular boundary arc and unbounded aspect ratios. We then derive $H^1$- and $L^2$-error estimates for the conforming discretisation of the Poisson problem. In the numerical experiments for $k=1,2$, the aspect ratio increases both under refinement and at a fixed horizontal mesh width. We also check the effect of numerical quadrature.

The remainder of the paper is organised as follows. Section~\ref{sec:curved-triangulations} introduces the geometry, the mapped Lagrange spaces, the geometric assumptions, and a class of anisotropic exact-curved families generated by a global map. Section~\ref{sec:interpolation} establishes interpolation estimates of any fixed degree. Section~\ref{model=sec} applies them to the Poisson problem. Section~\ref{sec:numerical-algorithm} describes the assembly with the exact element maps, and Section~\ref{sec:numerical-experiment} presents numerical experiments. Section~\ref{sec:concluding-remarks} concludes the paper.

\section{Exact curved triangulations and Lagrange finite elements} \label{sec:curved-triangulations}

\subsection{Reference element and a two-step affine map} \label{ref=affine=subsec}
Let $\widehat{T} \subset \mathbb{R}^2$ be the reference triangle with vertices $\hat{p}_1 := (0,0)^{\top}$, $\hat{p}_2 := (1,0)^{\top}$, and $\hat{p}_3 := (0,1)^{\top}$. We use the two-step affine map of \cite{IshKobTsu23}. Let $T\subset\mathbb R^2$ be a triangle with vertices $p_1,p_2,p_3$ such that $\overline{p_2p_3}$ is the longest edge of $T$. We set $h_T:=|p_2-p_3|$, $h_1:=|p_1-p_2|$, and $h_2:=|p_1-p_3|$, and we assume that $h_2\leq h_1$; then $\frac12h_T<h_1\leq h_T$. With
\begin{align}
\displaystyle
\widehat{A} := \diag (h_1,h_2),
\quad
\widetilde{A} :=
\begin{pmatrix}
1 & s \\
0 & t \\
\end{pmatrix},
\quad s^2+t^2=1,\ t>0, \label{aff=3}
\end{align}
we set ${A}_{\widetilde{T}} := \widetilde {A} \widehat{A}$ and define
\begin{align}
\displaystyle
\Phi_{\widetilde{T}}: \widehat{T} \ni \hat{x} \mapsto \tilde{x} := {A}_{\widetilde{T}} \hat{x} \in  \widetilde{T}. \label{aff=1}
\end{align}
The triangle $\widetilde T$ has the vertices $\tilde{p}_1 := (0,0)^{\top}$, $\tilde{p}_2 := (h_1,0)^{\top}$, and $\tilde{p}_3 :=(h_2 s , h_2 t)^{\top}$. Let $b_{T} \in \mathbb{R}^2$ and the rotation and mirror-imaging matrix ${A}_{T} \in O(2)$ satisfy $p_i = {A}_{T} \tilde{p}_i + b_T$ for $i=1,2,3$, and define
\begin{align}
\displaystyle
\Phi_{T}: \widetilde{T} \ni \tilde{x} \mapsto y := {A}_{T} \tilde{x} + b_{T} \in T. \label{aff=2}
\end{align}
The affine map from the reference triangle to $T$ is
\begin{align}
\displaystyle
\Phi := {\Phi}_{T} \circ {\Phi}_{\widetilde{T}}: \widehat{T} \ni \hat{x} \mapsto y = {A} \hat{x} + b_{T} \in T,
\quad {A} := {A}_{T} {A}_{\widetilde{T}}. \label{affinemap}
\end{align}

\subsection{Exact curved triangles} \label{subsec:exact-curved-triangles}
Let $T\subset\mathbb R^2$ be an affine triangle obtained from the affine map \eqref{affinemap}. We call $T$ the affine core to distinguish it from the curved physical element $K$. An exact curved triangle associated with $T$ is a set $K=\Psi(T)$, where $\Psi:U_T\to\Psi(U_T)$ is a $\mathcal C^2$-diffeomorphism on an open neighbourhood $U_T$ of $T$ with $\det D\Psi(y)>0$ for all $y\in T$; here $D\Psi(y)=(\partial\Psi_i/\partial y_j(y))_{i,j=1}^2$ is the Jacobian matrix. The element map is $F:=\Psi\circ\Phi:\widehat T\to K$. If $K$ is an interior element, we allow $\Psi=\operatorname{id}$ on $T$, so that $K=T$. If $K$ is a boundary element, then $\Psi$ is chosen so that $\Psi(e)\subset\partial\Omega$ for any affine edge $e\subset\partial T$ corresponding to a boundary edge. Thus, the mesh represents the physical boundary exactly, rather than by a polygonal approximation.

\subsection{Exact curved triangulations}
\label{subsec:exact-curved-triangulations}
Let $\Omega\subset\mathbb R^2$ be a bounded domain with piecewise smooth boundary.  In covering and intersection statements, elements include their boundaries; Sobolev spaces and integrals use their interiors.  An exact curved triangulation of $\Omega$ is a finite family
\begin{align*}
\displaystyle
\mathbb{K}_h = \{ K \}
\end{align*}
of exact curved triangles such that the following conditions hold. First, the elements cover the physical domain exactly:
\begin{align}
\displaystyle
\overline{\Omega}
= \bigcup_{K \in \mathbb K_h} K.
        \label{eq:exact-covering}
\end{align}
Second, distinct elements have disjoint interiors:
\begin{align}
\displaystyle
 \operatorname{int}K_1\cap\operatorname{int}K_2=\emptyset
        \quad
        \text{if } K_1,K_2 \in \mathbb K_h,\ K_1 \ne K_2.
        \label{eq:disjoint-interiors}
\end{align}
Third, the intersection of two distinct elements is empty, a common vertex, or a common curved edge.  In particular, if $K_1\cap K_2$ is one-dimensional, then it is represented by the same geometric curve from both sides. 

For each $K\in\mathbb K_h$, let $T_K$ be its affine core, so that $K=\Psi_K(T_K)$, and let $\Phi_K^{\mathrm{aff}}:=\Phi_{T_K}\circ\Phi_{\widetilde T_K}:\widehat T\to T_K$ be the two-step affine map of Section~\ref{ref=affine=subsec}. The element map is $F_K:=\Psi_K\circ\Phi_K^{\mathrm{aff}}:\widehat T\to K$. When no confusion can arise, we write simply $\Phi$ for the affine map from $\widehat T$ to the affine core and $F=\Psi\circ\Phi$ for the complete element map.

We also require compatibility of the edge parametrisations. Let $e=K\cap K'$ be a common edge of two distinct elements, and let $\hat e_K$ and $\hat e_{K'}$ be the corresponding edges of the reference triangles. There exists an affine bijection $R_e:\hat e_K\to\hat e_{K'}$, which maps corresponding endpoints to each other and satisfies
\begin{align}
F_K|_{\hat e_K}
=
F_{K'}|_{\hat e_{K'}}\circ R_e.
\label{eq:edge-parametrisation-compatibility}
\end{align}
The affine map $R_e$ includes a possible reversal of edge orientation. Thus, adjacent element maps induce the same parametrisation of a common physical edge up to orientation.

\begin{remark}
In conforming isoparametric meshes, this compatibility follows from the shared geometric nodes. Here, the curved corrections $\Psi_K$ are given elementwise, and we therefore state the compatibility of the edge parametrisations as a separate condition.
\end{remark}

The collection of affine cores is denoted by $\mathbb T_h:=\{T_K:K\in\mathbb K_h\}$; its union need not be close to $\Omega$. The global mesh size $h:=\max_{K\in\mathbb K_h}h_{T_K}$, where $h_{T_K}$ is the length of the longest edge of $T_K$, is measured on the affine cores, not directly on the curved physical elements.

Let $\mathcal E_h^\partial$ be the set of boundary edges of $\mathbb K_h$, namely the curved edges which are not shared by two distinct elements.  We require the exact boundary representation:
\begin{align}
\displaystyle
\partial \Omega
= \bigcup_{e \in \mathcal E_h^\partial} e. \label{eq:exact-boundary}
\end{align}

\begin{remark}
The exactness in \eqref{eq:exact-covering} and \eqref{eq:exact-boundary} is geometric.  It does not mean that the element maps are polynomial, nor that the geometry is described by an isoparametric ansatz.  The curved maps $F_K=\Psi_K\circ\Phi_K^{\mathrm{aff}}$ are regarded as part of the mesh data.
\end{remark}

\subsection{Lagrange finite element spaces}
\label{subsec:lagrange-fe-spaces}
Let $k \geq 1$ be an integer.  On the reference triangle $\widehat T$, we denote by $\mathbb{P}^k(\widehat T)$ the space of polynomials of total degree at most $k$.  Let $\{\hat a_i\}_{i=1}^{N_k}$, $N_k=(k+1)(k+2)/2$, be the standard Lagrange nodes of degree $k$ on $\widehat T$, that is, the points $\sum_{j=1}^3(\alpha_j/k)\hat p_j$ with $\alpha\in\mathbb N_0^3$ and $\alpha_1+\alpha_2+\alpha_3=k$. The Lagrange nodes on $K\in\mathbb K_h$ are $a_{K,i}:=F_K(\hat a_i)$, $i=1,\ldots,N_k$. Thus, the nodes on a boundary edge of $K$ lie on the exact curved boundary whenever the corresponding reference nodes lie on the associated reference edge.

The local exact-curved Lagrange space of degree $k$ on $K$ is defined as
\begin{align}
\displaystyle
V_{L}^k(K)
:= \left\{ \hat v \circ F_K^{-1}:\ \hat v \in \mathbb{P}^k(\widehat T) \right\}. \label{eq:local-curved-lagrange-space}
\end{align}
The space is obtained by transporting $\mathbb{P}^k(\widehat T)$ through $F_K$. If $F_K$ is non-affine, its functions are generally not polynomials in the physical coordinates.

Let $I_{\widehat{T}}^k : \mathcal{C}^0(\widehat T) \to \mathbb{P}^k(\widehat T)$ be the standard Lagrange interpolation operator on the reference triangle. For a continuous function $v  \in \mathcal{C}^0(K)$, we define the local interpolation
operator $I_{L}^k : \mathcal{C}^0(K) \to V_L^k(K)$ as
\begin{align}
\displaystyle
I_L^k v
:= \left(I_{\widehat{T}}^k(v\circ F_K)\right) \circ F_K^{-1}. \label{eq:local-interpolation-operator}
\end{align}
Then, $(I_L^k v)(a_{K,i}) = v(a_{K,i})$ for $i=1,\ldots,N_k$; thus, $I_L^k$ is the nodal Lagrange interpolation operator on the exact curved triangle $K$. When the element must be indicated explicitly, we write this operator as $I_{L,K}^k$.

The global conforming Lagrange finite element space is defined as
\begin{align}
\displaystyle
V_h^k
:=
\left\{
v_h \in \mathcal{C}^0(\overline{\Omega}): \ v_h|_K \in V_{L}^k(K) \quad \forall K\in\mathbb K_h \right\}. \label{eq:global-lagrange-space}
\end{align}
For homogeneous Dirichlet boundary conditions, we set
\begin{align}
\displaystyle
V_{h0}^k
:= V_h^k \cap H_0^1(\Omega). \label{eq:global-lagrange-space-zero}
\end{align}
Since the boundary edges of $\mathbb K_h$ lie on $\partial\Omega$, the condition $v_h=0$ is imposed on the physical boundary itself.

For $v\in\mathcal C^0(\overline\Omega)$, define the global nodal interpolation elementwise by
\begin{align}
(I_h^k v)|_K:=I_{L,K}^k(v|_K),\quad K\in\mathbb K_h.
\label{eq:global-nodal-interpolation}
\end{align}

\begin{lemma}[Compatibility of the global interpolation]
\label{lem:global-interpolation-compatibility}
Under the edge-parametrisation condition \eqref{eq:edge-parametrisation-compatibility}, the function $I_h^k v$ defined by \eqref{eq:global-nodal-interpolation} belongs to $V_h^k$. If, in addition, $v=0$ on $\partial\Omega$, then $I_h^k v\in V_{h0}^k$.
\end{lemma}

\begin{proof}
Let $e=K\cap K'$ be a common edge. The traces of $I_{L,K}^k(v|_K)$ and $I_{L,K'}^k(v|_{K'})$, pulled back to $\hat e_K$ by $F_K$ and $F_{K'}\circ R_e$, respectively, are one-dimensional polynomials of degree at most $k$. By \eqref{eq:edge-parametrisation-compatibility}, the corresponding $k+1$ Lagrange nodes have the same physical locations. Because $v$ is continuous, the two traces have the same values at these nodes and hence coincide on the whole edge. Therefore, $I_h^k v$ is continuous across any common edge and belongs to $V_h^k$. If $v=0$ on $\partial\Omega$, all boundary Lagrange nodes vanish, and the polynomial trace on each boundary edge is identically zero. Therefore, $I_h^k v\in V_{h0}^k$.
\end{proof}

\subsection{Additional notation and geometric assumptions}
\label{subsec:geometric-assumptions}
In this subsection, we state the geometric assumptions used in the interpolation analysis. The shape condition is imposed on the affine cores $T_K\in\mathbb T_h$, which carry the anisotropic scaling, and not on the curved elements $K\in\mathbb K_h$; the curved corrections $\Psi_K$ are controlled separately. Let $K \in \mathbb K_h$, and let $T_K\in\mathbb T_h$ be its affine core.  We label the vertices $p_1,p_2,p_3$ of $T_K$ as in Section~\ref{ref=affine=subsec} and write $h_{T_K}:=|p_2-p_3|$, $h_{K,1}:=|p_1-p_2|$, and $h_{K,2}:=|p_1-p_3|$, so that $h_{K,2}\leq h_{K,1}\simeq h_{T_K}$. We define the unit vectors $r_1:=(p_2-p_1)/|p_2-p_1|$ and $r_2:=(p_3-p_1)/|p_3-p_1|$. For $x\in K$, let $y = \Psi_K^{-1}(x) \in T_K$. The affine-core direction $r_i$ is transported to a vector field on $K$ by
\begin{align}
\displaystyle
\tau_{K,i}(x)
:= D\Psi_K(y) r_i
= D\Psi_K(\Psi_K^{-1}(x)) r_i, \quad i=1,2 . \label{eq:transported-direction}
\end{align}
We call $\tau_{K,i}$ the transported direction associated with $r_i$. It is not necessarily a unit vector.

For sufficiently smooth functions $w$ on $T_K$ and $v$ on $K$, we define the directional derivatives $\partial w/\partial r_i:=r_i\cdot\nabla_y w$ and $\partial v/\partial\tau_{K,i}:=\tau_{K,i}\cdot\nabla_x v$, where $y$ and $x$ denote the coordinates on $T_K$ and $K$, respectively. We also write
\begin{align*}
&\mathcal D_{K,i}v:=\frac{\partial v}{\partial \tau_{K,i}},
\quad
\mathcal D_K^{\alpha}
:=\mathcal D_{K,1}^{\alpha_1}\mathcal D_{K,2}^{\alpha_2},
\\
&\partial_r^{\alpha}
:=\left(\frac{\partial}{\partial r_1}\right)^{\alpha_1}
  \left(\frac{\partial}{\partial r_2}\right)^{\alpha_2},
\quad
h_K^{\alpha}:=h_{K,1}^{\alpha_1}h_{K,2}^{\alpha_2}
\end{align*}
for $\alpha=(\alpha_1,\alpha_2)\in\mathbb N_0^2$.
For $i,j=1,2$, we also define the non-affine correction field
\begin{align}
\displaystyle
b_{K,ij}(x)
:= D^2\Psi_K(\Psi_K^{-1}(x))[r_i,r_j], \quad x \in K. \label{eq:curvature-vector-field}
\end{align}

We define the affine geometric parameter associated with the affine core $T_K$ as
\begin{align}
\displaystyle
 H_{T_K} := \frac{h_{K,1} h_{K,2}}{|T_K|_2} h_{T_K}, \label{eq:HTK-definition}
\end{align}
where  $|T_K|_2$ denotes the two-dimensional Hausdorff measure of $T_K$.

\begin{assumption}[Affine semi-regularity]
\label{ass:affine-semi-regularity}
There exists a constant $\gamma_0>0$, independent of $h$, such that
\begin{align}
\displaystyle
\frac{H_{T_K}}{h_{T_K}} \leq \gamma_0 \quad \forall K \in \mathbb K_h. \label{eq:affine-semi-regularity}
\end{align}
Related geometric conditions for anisotropic simplices can be found in \cite{IshKobTsu21a,IshKobTsu23,IshKobSuzTsu21}.
\end{assumption}

\begin{assumption}[Uniform control of the curved correction]
\label{ass:curved-correction}
There exist constants $C_\Psi^{(1)},C_\Psi^{(2)}>0$, independent of both $h$ and $K$, such that
\begin{align*}
\displaystyle
\|D\Psi_K\|_{L^\infty(T_K)} + \|D\Psi_K^{-1}\|_{L^\infty(K)} \leq C_\Psi^{(1)},
\end{align*}
and
\begin{align*}
\displaystyle
\|D^2\Psi_K\|_{L^\infty(T_K)} + \|D^2\Psi_K^{-1}\|_{L^\infty(K)} \leq C_\Psi^{(2)}.
\end{align*}
\end{assumption}

\begin{assumption}[Higher-order curved-map control]
\label{ass:curved-correction-higher}
For a fixed integer $k\ge2$, each $\Psi_K$ extends to a $\mathcal C^{k+1}$-diffeomorphism on a neighbourhood of $T_K$.  For any $3\le j\le k+1$, there is a constant $C_{\Psi,j}$ independent of $h$ and $K$ such that
\begin{align*}
\|D^j\Psi_K\|_{L^\infty(T_K)}\le C_{\Psi,j}.
\end{align*}
Together with Assumption~\ref{ass:curved-correction}, this controls all forward-map derivatives needed for degree $k$.  Constants may depend on the fixed degree $k$; no bound uniform as $k\to\infty$ is asserted.
\end{assumption}

\subsection{Families generated by a global map}
\label{subsec:explicit-exact-family}
The assumptions of Section~\ref{subsec:geometric-assumptions} do not require uniformly bounded aspect ratios. The following construction gives one class of anisotropic exact curved triangulations satisfying them.

\begin{proposition}[Exact curved triangulations generated by a global map]
\label{prop:global-map}
Let $P\subset\mathbb R^2$ be a bounded polygonal domain, let $m\geq2$ be an integer, and let $\Psi$ be a $\mathcal C^m$-diffeomorphism from an open neighbourhood $U$ of $\overline P$ onto $\Psi(U)$ with $\det D\Psi>0$ on $U$. Set $\Omega:=\Psi(P)$. Let $\{\mathbb T_h\}_h$ be a family of conforming triangulations of $\overline P$ whose triangles satisfy \eqref{eq:affine-semi-regularity} with a constant $\gamma_0$. We set
\begin{align*}
\mathbb K_h:=\{\Psi(T):T\in\mathbb T_h\},
\end{align*}
and take $\Psi_K:=\Psi$ for any $K\in\mathbb K_h$. Then, each $\mathbb K_h$ is an exact curved triangulation of $\Omega$ satisfying \eqref{eq:exact-covering}, \eqref{eq:disjoint-interiors}, \eqref{eq:exact-boundary}, and \eqref{eq:edge-parametrisation-compatibility} for any degree $k\geq1$. Furthermore, Assumption~\ref{ass:curved-correction} holds, and Assumption~\ref{ass:curved-correction-higher} holds for any $k\leq m-1$, with constants depending only on $\Psi$ and $P$.
\end{proposition}

\begin{proof}
Because $\Psi$ is a homeomorphism of $U$ onto the open set $\Psi(U)$, it maps $\overline P$ onto $\overline\Omega$ and $\partial P$ onto $\partial\Omega$, and $\Psi(T)\cap\Psi(T')=\Psi(T\cap T')$ for $T,T'\in\mathbb T_h$. Therefore, the elements of $\mathbb K_h$ cover $\overline\Omega$, have disjoint interiors, and intersect, if at all, in a common vertex or a common curved edge. The boundary edges of $\mathbb K_h$ are the images of the boundary edges of $\mathbb T_h$, which cover $\partial P$; thus \eqref{eq:exact-boundary} holds.

Let $e=K\cap K'$ be a common edge, where $K=\Psi(T)$ and $K'=\Psi(T')$. The restrictions of the affine maps of $T$ and $T'$ to the corresponding reference edges are affine parametrisations of the same segment $T\cap T'$. Therefore, $\Phi_K^{\mathrm{aff}}|_{\hat e_K}=\Phi_{K'}^{\mathrm{aff}}|_{\hat e_{K'}}\circ R_e$ for the affine bijection $R_e$ that maps corresponding endpoints to each other. Composing both sides with $\Psi$ gives \eqref{eq:edge-parametrisation-compatibility}.

The derivatives of $\Psi$ of order at most $m$ are continuous on $U$ and hence bounded on the compact set $\overline P$. By the inverse function theorem, $\Psi^{-1}$ is a $\mathcal C^m$ map on $\Psi(U)$, and its first and second derivatives are bounded on the compact set $\overline\Omega$. Since each $\Psi_K$ is the restriction of $\Psi$, these bounds are uniform in $h$ and $K$.
\end{proof}

\begin{remark}
\label{rem:global-map-scope}
Proposition~\ref{prop:global-map} assumes a single smooth diffeomorphism defined on the whole domain. Such a map preserves boundary corners and therefore cannot transform a polygon into a domain with an entirely smooth boundary, such as a disc. Nevertheless, it allows curved boundary segments, including circular arcs, provided that the resulting domain retains corners. The construction used in our numerical experiments is an example. This proposition does not address elementwise boundary corrections. The interpolation results below nevertheless allow element-dependent maps whenever the stated compatibility and uniform bounds hold.
\end{remark}

\subsubsection*{An explicit family with an exact circular arc}
As an explicit instance, we set $a=1/2$, $g(s)=\sqrt{1-s^2}$, and
\begin{align}
Q&=(-a,a)\times(0,1),\quad
\Psi(s,t)=(s,tg(s)),\notag\\
\Omega&=\Psi(Q)
=\{(x,y):-a<x<a,\ 0<y<g(x)\}.
\label{eq:cap-domain-map}
\end{align}
The upper boundary $y=g(x)$ is an exact arc of the unit circle. Because $g''(s)=-(1-s^2)^{-3/2}<0$, the domain $\Omega$ is convex.

For an integer $N\ge2$, let $\delta=1/N$. We divide $[-a,a]$ into $N$ intervals of length $\delta$ and $[0,1]$ into $N^2$ intervals of length $\delta^2$.  We split each resulting closed rectangle of size $\delta\times\delta^2$ along the diagonal from its lower-left to its upper-right vertex, except for the top-left and bottom-right rectangles, which use the opposite diagonal.  Equivalently, if rectangles are indexed by $(i,j)\in\{0,\ldots,N-1\}\times\{0,\ldots,N^2-1\}$ from the bottom-left, the opposite diagonal is used precisely for $(i,j)=(0,N^2-1)$ and $(N-1,0)$.  These two changes prevent a core triangle from having all three vertices on $\partial Q$. We denote the conforming affine triangulation by $\mathbb T_N$, and let $\mathbb K_N$ be the exact curved triangulation given by Proposition~\ref{prop:global-map} with $P=Q$.

Each affine core is a right triangle with legs $\delta$ and $\delta^2$, area $|T|_2=\delta^3/2$, and longest edge $h_T=\delta\sqrt{1+\delta^2}$.  Choosing the right-angle vertex as $p_1$ gives $h_{K,1}=\delta$, $h_{K,2}=\delta^2$, and therefore
\begin{align}
\frac{H_T}{h_T}
=\frac{h_{K,1}h_{K,2}}{|T|_2}=2
\quad\text{for any }T\in\mathbb T_N,
\quad
\frac{h_{K,1}}{h_{K,2}}=N\longrightarrow\infty.
\label{eq:cap-semi-regularity}
\end{align}
Furthermore,
\begin{align*}
D\Psi(s,t)
&=\begin{pmatrix}1&0\\tg'(s)&g(s)\end{pmatrix},
\quad \det D\Psi(s,t)=g(s)\ge\sqrt{3}/2,\\
\Psi^{-1}(x,y)&=(x,y/g(x)).
\end{align*}
The map $\Psi$ is a $\mathcal C^\infty$-diffeomorphism from the open neighbourhood $(-1,1)\times\mathbb R$ of $\overline Q$ onto its image. Proposition~\ref{prop:global-map} therefore applies with any $m$: Assumptions~\ref{ass:affine-semi-regularity}, \ref{ass:curved-correction}, and \ref{ass:curved-correction-higher} hold with $\gamma_0=2$ and map constants independent of $N$, while the aspect ratio $N$ tends to infinity.  The domain has a circular top boundary and straight sides and bottom; it is not the full disc.

Figure~\ref{fig:exact-cap-mesh} illustrates the $N=4$ instance used in the numerical experiments.

\begin{figure}[htbp]
\centering
\includegraphics[width=\textwidth]{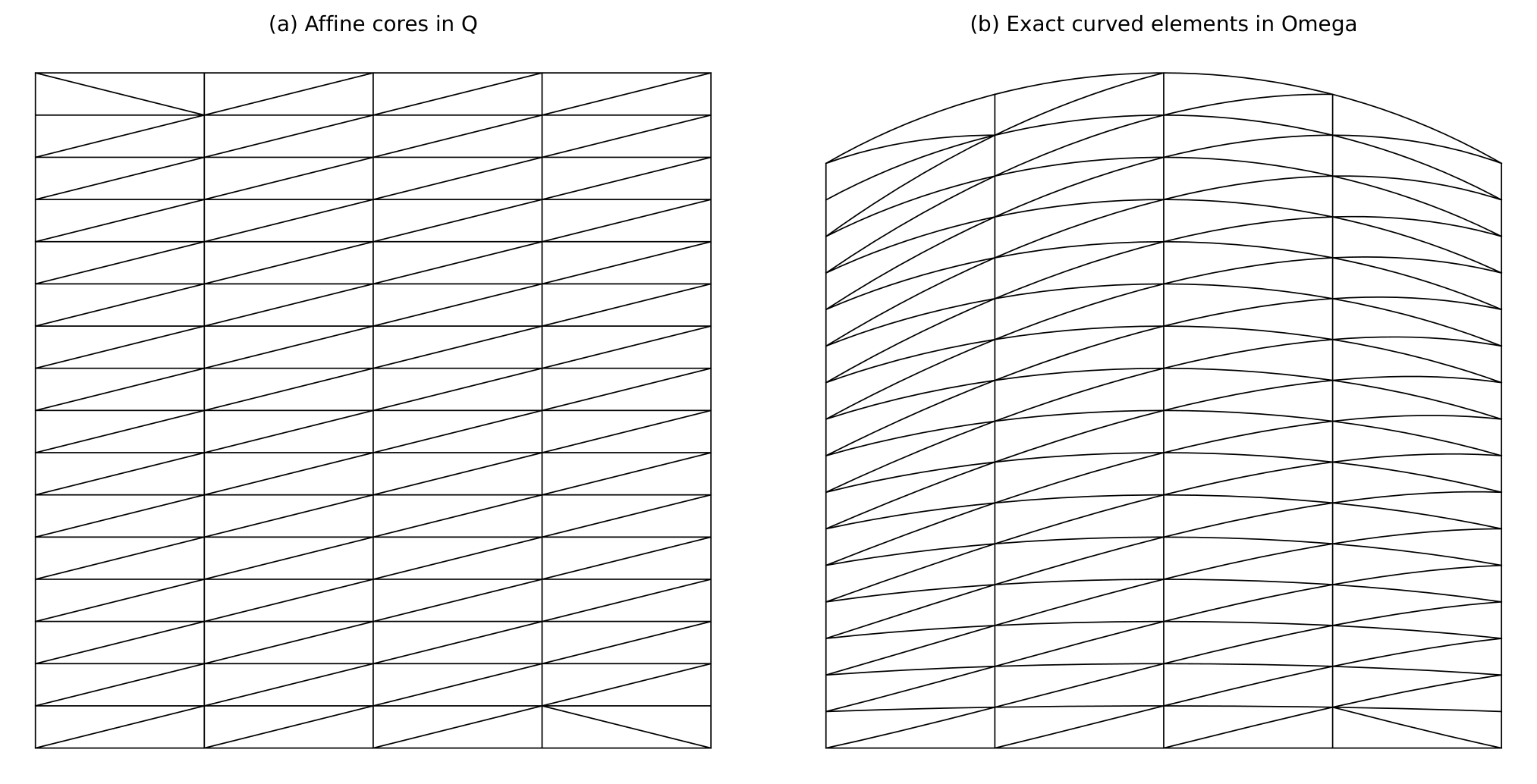}
\caption{The mesh for $N=4$: (a) the affine cores in $Q$, split as described in Section~\ref{subsec:explicit-exact-family}; (b) the 128 curved elements obtained by $\Psi(s,t)=(s,t\sqrt{1-s^2})$.  The upper boundary is the exact unit-circle arc.  The map also deforms interior edges: each horizontal line $t=t_0$ is mapped to $y=t_0\sqrt{1-x^2}$.}
\label{fig:exact-cap-mesh}
\end{figure}

\section{Interpolation on exact curved elements}
\label{sec:interpolation}
The analysis is related to the curved-element interpolation theory of Bernardi \cite{Ber89}, but we work with a factorised element map rather than treating it as a perturbation of an affine map. We use the factorisation
\begin{align*}
\displaystyle
F_K=\Psi_K\circ\Phi_K^{\mathrm{aff}},
\end{align*}
which separates the affine core $T_K$ from the curved correction $\Psi_K$. We apply the anisotropic interpolation estimate on the affine core and transfer the result to $K$ by $\Psi_K$.

Let $K \in \mathbb K_h$, and let $T_K \in \mathbb T_h$ be its affine core.  We recall
\begin{align*}
\displaystyle
F_K=\Psi_K\circ\Phi_K^{\mathrm{aff}}:\widehat T \to K .
\end{align*}
For a scalar function $v$ on $K$, we introduce
\begin{align*}
\displaystyle
w := v\circ \Psi_K \quad \text{on }T_K.
\end{align*}
Under Assumption~\ref{ass:curved-correction}, composition with $\Psi_K$ maps $H^2(K)$ continuously into $H^2(T_K)$. In two dimensions, $H^2$ functions on a Lipschitz triangle have continuous representatives, so the nodal values used by the Lagrange interpolants are well defined. Identities between weak derivatives below are understood almost everywhere when the function is only in a Sobolev space.
Then, 
\begin{align*}
\displaystyle
v\circ F_K
= (v\circ\Psi_K)\circ\Phi_K^{\mathrm{aff}} = w\circ\Phi_K^{\mathrm{aff}}.
\end{align*}
Therefore, by the definition of the exact-curved Lagrange interpolation operator,
\begin{align}
\displaystyle
I_L^k v = (I_{T_K}^k w)\circ\Psi_K^{-1}, \label{eq:interpolation-reduction}
\end{align}
where $I_{T_K}^k$ denotes the usual Lagrange interpolation operator on the affine core $T_K$.  Consequently,
\begin{align}
\displaystyle
v-I_L^k v
= (w-I_{T_K}^k w)\circ\Psi_K^{-1}. \label{eq:error-reduction}
\end{align}

\subsection{Transfer to the affine core}
\label{subsec:scaling-argument}
\begin{lemma}[Transfer of $L^2$- and $H^1$-norms]
\label{lem:L2-H1-transfer}
Let $v \in H^1(K)$ and $w := v \circ \Psi_K$. Then,
\begin{align}
\displaystyle
\|v\|_{L^2(K)}
\leq C_\Psi^{(3)} \|w\|_{L^2(T_K)},
\quad
|v|_{H^1(K)}
\leq C_\Psi^{(4)} |w|_{H^1(T_K)}, \label{eq:L2-transfer}
\end{align}
and
\begin{align}
\displaystyle
\|w\|_{L^2(T_K)}
\leq C_\Psi^{(3)} \|v\|_{L^2(K)},
\quad
|w|_{H^1(T_K)}
\leq C_\Psi^{(4)} |v|_{H^1(K)}. \label{eq:reverse-transfer}
\end{align}
The constants $C_\Psi^{(3)}$ and $C_\Psi^{(4)}$ depend only on the bounds for $D\Psi_K$ and $D\Psi_K^{-1}$ in Assumption~\ref{ass:curved-correction}.
\end{lemma}

\begin{proof}
By density of $\mathcal{C}^{1}(\overline K)$ in $H^{1}(K)$, it suffices to consider $v\in\mathcal{C}^1(\overline K)$. The change of variables $x=\Psi_K(y)$ and the chain rule $\nabla_x v(x)= D\Psi_K^{-1}(x)^{\top}\nabla_y w(y)$ give
\begin{align*}
\displaystyle
\|v\|_{L^2(K)}^2
&= \int_{T_K}|w|^2 |\det D\Psi_K| \, dy,
\\
|v|_{H^1(K)}^2
&\leq \|D\Psi_K^{-1}\|_{L^\infty(K)}^2 \int_{T_K} |\nabla_y w|^2 |\det D\Psi_K| \, dy.
\end{align*}
Since $|\det D\Psi_K|\leq\|D\Psi_K\|_{L^\infty(T_K)}^2$, this proves \eqref{eq:L2-transfer}. Interchanging the roles of $\Psi_K$ and $\Psi_K^{-1}$ proves \eqref{eq:reverse-transfer}.
\end{proof}

\begin{lemma}[Transport of directional derivatives]
\label{lem:directional-derivative-transfer}
Let $s\geq1$ be an integer, suppose that $\Psi_K$ is of class $\mathcal C^{s}$ on a neighbourhood of $T_K$, let $v\in H^s(K)$, and set $w:=v\circ\Psi_K$. Then, for $i=1,2$ and any $g\in H^1(K)$,
\begin{align}
\displaystyle
\frac{\partial (g\circ\Psi_K)}{\partial r_i}
= \left(\mathcal D_{K,i}g\right)\circ\Psi_K
\quad \text{a.e. on }T_K. \label{eq:first-directional-transfer}
\end{align}
Consequently, $\mathcal D_{K,1}$ and $\mathcal D_{K,2}$ commute on $H^2(K)$ when $s\geq2$, and
\begin{align}
\displaystyle
\partial_r^{\alpha}w
=\bigl(\mathcal D_K^{\alpha}v\bigr)\circ\Psi_K
\quad\text{for } |\alpha|\leq s. \label{eq:high-order-transport}
\end{align}
Furthermore, with $y=\Psi_K^{-1}(x)$,
\begin{align}
\displaystyle
D\tau_{K,i}(x)[\xi]
=
D^2\Psi_K(y)[D\Psi_K^{-1}(x)\xi,r_i],
\quad \xi\in\mathbb R^2,\ x\in K, \label{eq:transported-direction-derivative}
\end{align}
and, if $s\geq2$, for $i,j=1,2$,
\begin{align}
\displaystyle
\mathcal D_{K,i}\mathcal D_{K,j}v
= D^2v[\tau_{K,i},\tau_{K,j}] + \nabla_x v\cdot b_{K,ij}
\quad \text{a.e. on }K. \label{eq:second-directional-transfer-K}
\end{align}
\end{lemma}

\begin{proof}
For $g\in\mathcal C^1(\overline K)$, the chain rule gives $\nabla_y(g\circ\Psi_K)(y) = D\Psi_K(y)^\top \nabla_x g(\Psi_K(y))$. Therefore, by the definition \eqref{eq:transported-direction} of the transported direction,
\begin{align*}
\displaystyle
\frac{\partial (g\circ\Psi_K)}{\partial r_i}(y)
= D\Psi_K(y)r_i\cdot\nabla_x g(\Psi_K(y))
= \left(\tau_{K,i}\cdot\nabla_x g\right)(\Psi_K(y)).
\end{align*}
The general case of \eqref{eq:first-directional-transfer} follows by density. Because $\tau_{K,i}=(D\Psi_K\circ\Psi_K^{-1})r_i$ is of class $\mathcal C^{s-1}$, we have $\mathcal D_K^{\beta}v\in H^{s-|\beta|}(K)$ for $|\beta|\leq s$; hence \eqref{eq:first-directional-transfer} applies to $g=\mathcal D_K^{\beta}v$ whenever $|\beta|\leq s-1$, and induction on $|\alpha|$ gives \eqref{eq:high-order-transport}. The commutativity of $\mathcal D_{K,1}$ and $\mathcal D_{K,2}$ follows from that of $\partial/\partial r_1$ and $\partial/\partial r_2$. Differentiating the identity $\tau_{K,i}(\Psi_K(y))=D\Psi_K(y)r_i$ with respect to $y$ in a direction $\eta$ gives
\begin{align*}
\displaystyle
D\tau_{K,i}(\Psi_K(y))D\Psi_K(y)\eta=D^2\Psi_K(y)[\eta,r_i],
\end{align*}
and setting $\eta=D\Psi_K^{-1}(x)\xi$ yields \eqref{eq:transported-direction-derivative}. Finally, the product rule gives
\begin{align*}
\displaystyle
\mathcal D_{K,i}\mathcal D_{K,j}v
= D^2v[\tau_{K,i},\tau_{K,j}]
+ \nabla_x v\cdot D\tau_{K,j}[\tau_{K,i}],
\end{align*}
and \eqref{eq:transported-direction-derivative} with $\xi=\tau_{K,i}(x)=D\Psi_K(y)r_i$ gives 
\begin{align*}
\displaystyle
D\tau_{K,j}(x)[\tau_{K,i}(x)]=D^2\Psi_K(y)[r_i,r_j]=b_{K,ij}(x).
\end{align*}
\end{proof}

The second term in \eqref{eq:second-directional-transfer-K} comes from the non-affine part of the element map. It may reflect boundary curvature, but it can also be non-zero under a non-linear reparametrisation of a straight element.

\begin{lemma}[Composition with the curved correction]
\label{lem:sobolev-composition}
Let $s\geq1$ be an integer. Suppose that Assumption~\ref{ass:curved-correction} holds and that the derivatives of $\Psi_K$ of order at most $s$ are bounded on $T_K$ independently of $h$ and $K$. If $v\in H^s(K)$, then $w:=v\circ\Psi_K\in H^s(T_K)$ and
\begin{align}
\displaystyle
|w|_{H^j(T_K)}\leq C\sum_{i=1}^{j}|v|_{H^i(K)},
\quad 1\leq j\leq s, \label{eq:sobolev-composition}
\end{align}
where $C$ depends only on $s$ and on these bounds.
\end{lemma}

\begin{proof}
For $v\in\mathcal C^s(\overline K)$, the higher-order chain rule expresses $D^jw(y)$ as a sum of terms in which $D^iv(\Psi_K(y))$, $1\leq i\leq j$, is applied to $i$ derivatives of $\Psi_K$ at $y$ whose orders add up to $j$. The change of variables $x=\Psi_K(y)$ and $|\det D\Psi_K^{-1}|\leq\|D\Psi_K^{-1}\|_{L^\infty(K)}^2$ give \eqref{eq:sobolev-composition}, and the general case follows by density.
\end{proof}

\subsection{Interpolation of any fixed degree}
\label{subsec:high-order-interpolation}
Fix an integer $k\ge1$.  The affine-core estimates of \cite{IshKobTsu23} are stated in terms of the directional derivatives $\partial_r^\alpha$.  By \eqref{eq:high-order-transport}, these correspond to the derivatives $\mathcal D_K^\alpha$ along the transported directions on $K$.  For $|\alpha|=2$, the identity \eqref{eq:second-directional-transfer-K} shows that $\mathcal D_K^\alpha$ contains the non-affine correction field $b_{K,ij}$.

\begin{theorem}[Fixed-degree interpolation on exact curved triangles]
\label{thm:high-order-curved}
Suppose that Assumptions~\ref{ass:affine-semi-regularity} and \ref{ass:curved-correction} hold, as well as Assumption~\ref{ass:curved-correction-higher} when $k\ge2$. For any integer $2\le s\le k+1$ and $v\in H^s(K)$, the degree-$k$ mapped Lagrange interpolant satisfies
\begin{align}
\|v-I_{L,K}^kv\|_{L^2(K)}
&\le C_{k,s}\sum_{|\alpha|=s}h_K^{\alpha}
       \|\mathcal D_K^{\alpha}v\|_{L^2(K)},
\label{eq:high-order-curved-L2}\\
|v-I_{L,K}^kv|_{H^1(K)}
&\le C_{k,s}\frac{H_{T_K}}{h_{T_K}}
       \sum_{|\alpha|=s-1}h_K^{\alpha}
       |\mathcal D_K^{\alpha}v|_{H^1(K)}.
\label{eq:high-order-curved-H1}
\end{align}
The constants are independent of $h$, $K$, and the aspect ratio of $T_K$, but depend on the fixed $k,s$ and on the uniform map bounds.
\end{theorem}

\begin{proof}
We set $w=v\circ\Psi_K$.  By Lemma~\ref{lem:sobolev-composition}, $w\in H^s(T_K)$.  In two dimensions, \cite[Corollary~1, estimate~(58)]{IshKobTsu23} applies to the degree-$k$ Lagrange interpolant on $T_K$ with $p=q=2$, $\ell=s$, and $m=0$ or $1$; the conditions $0\leq m\leq\ell\leq k+1$ and $\ell-m\geq1$ required there hold for $2\leq s\leq k+1$. With $p=q=2$, the factor $|T_K|_2^{1/q-1/p}$ in that estimate equals one, and the edge lengths and directions used there are $h_{K,1},h_{K,2}$ and $r_1,r_2$.  It gives
\begin{align}
\|w-I_{T_K}^kw\|_{L^2(T_K)}
&\le C_{k,s}\sum_{|\alpha|=s}h_K^\alpha
   \|\partial_r^\alpha w\|_{L^2(T_K)},
\label{eq:L2-affine-core}\\
|w-I_{T_K}^kw|_{H^1(T_K)}
&\le C_{k,s}\frac{H_{T_K}}{h_{T_K}}
   \sum_{|\alpha|=s-1}h_K^\alpha
   |\partial_r^\alpha w|_{H^1(T_K)},
\label{eq:H1-affine-core}
\end{align}
where the constants are independent of the shape of $T_K$. By \eqref{eq:error-reduction} and \eqref{eq:L2-transfer}, the left-hand sides of \eqref{eq:high-order-curved-L2} and \eqref{eq:high-order-curved-H1} are bounded by $C_\Psi^{(3)}$ and $C_\Psi^{(4)}$ times the left-hand sides of \eqref{eq:L2-affine-core} and \eqref{eq:H1-affine-core}, respectively.  By \eqref{eq:high-order-transport}, $\partial_r^\alpha w=(\mathcal D_K^\alpha v)\circ\Psi_K$, and \eqref{eq:reverse-transfer} gives
\begin{align*}
\|\partial_r^\alpha w\|_{L^2(T_K)}
\le C_\Psi^{(3)}\|\mathcal D_K^\alpha v\|_{L^2(K)},
\quad
|\partial_r^\alpha w|_{H^1(T_K)}
\le C_\Psi^{(4)}|\mathcal D_K^\alpha v|_{H^1(K)}.
\end{align*}
Combining the estimates proves the result.
\end{proof}

\begin{remark}[The case $k=1$]
\label{rem:P1-explicit}
For $k=1$ and $s=2$, \eqref{eq:high-order-transport} and \eqref{eq:second-directional-transfer-K} show that Theorem~\ref{thm:high-order-curved} gives
\begin{align}
&\|v-I_L^1v\|_{L^2(K)} \\
&\hspace{-0.5cm} \leq C \sum_{i=1}^{2}\sum_{j=1}^{2} h_{K,i} h_{K,j}
\left(  \|D^2v[\tau_{K,i},\tau_{K,j}]\|_{L^2(K)}
+ \|\nabla_x v\cdot b_{K,ij}\|_{L^2(K)} \right),
\label{eq:L2-curved-P1}\\
&|v-I_L^1v|_{H^1(K)} 
\leq C \frac{H_{T_K}}{h_{T_K}} \sum_{i=1}^{2} h_{K,i}
\left| \frac{\partial v}{\partial \tau_{K,i}} \right|_{H^1(K)}.
\label{eq:H1-curved-P1}
\end{align}
If $\Psi_K$ is the identity, then $\tau_{K,i}=r_i$ and $b_{K,ij}=0$, and \eqref{eq:L2-curved-P1}--\eqref{eq:H1-curved-P1} reduce to the affine estimates of \cite{IshKobTsu23}.
\end{remark}

\begin{corollary}[Standard local rates]
\label{cor:high-order-local-standard}
Under the assumptions of Theorem~\ref{thm:high-order-curved},
\begin{align}
|v-I_{L,K}^kv|_{H^1(K)}
&\le C_{k,s}h_{T_K}^{s-1}\|v\|_{H^s(K)},
\label{eq:high-order-standard-H1}\\
\|v-I_{L,K}^kv\|_{L^2(K)}
&\le C_{k,s}h_{T_K}^{s}\|v\|_{H^s(K)}
\label{eq:high-order-standard-L2}
\end{align}
for any integer $2\le s\le k+1$ and $v\in H^s(K)$.
\end{corollary}

\begin{proof}
Because any $h_{K,i}\le h_{T_K}$ and $H_{T_K}/h_{T_K}\le\gamma_0$, \eqref{eq:error-reduction}, \eqref{eq:L2-transfer}, \eqref{eq:L2-affine-core}, and \eqref{eq:H1-affine-core} give
\begin{align*}
\|v-I_{L,K}^kv\|_{L^2(K)}\le C_{k,s}h_{T_K}^{s}|w|_{H^s(T_K)},
\quad
|v-I_{L,K}^kv|_{H^1(K)}\le C_{k,s}h_{T_K}^{s-1}|w|_{H^s(T_K)},
\end{align*}
with $w=v\circ\Psi_K$.  Lemma~\ref{lem:sobolev-composition} bounds $|w|_{H^s(T_K)}$ by $C\|v\|_{H^s(K)}$.
\end{proof}

\begin{remark}
\label{rem:full-norm}
The full $H^s$ norm in \eqref{eq:high-order-standard-H1}--\eqref{eq:high-order-standard-L2} cannot generally be replaced by the highest-order seminorm alone. For the non-affine map in \eqref{eq:cap-domain-map}, the physical affine function $v(x,y)=y$ has $|v|_{H^s(K)}=0$ for $s\ge2$, whereas $v\circ\Psi(s,t)=t\sqrt{1-s^2}$ is not in $\mathbb P^k(T_K)$ on any non-degenerate core. Therefore, its mapped interpolation error is non-zero.  The transported derivatives retain the necessary lower-order curvature terms.  The constants above may depend on $k$; no estimate uniform in $k$ is claimed.
\end{remark}

\section{Application to the Poisson problem} \label{model=sec}

\subsection{Model problem}
\label{subsec:model-problem}
Let $\Omega\subset\mathbb R^2$ be a bounded domain with piecewise smooth boundary. We consider the homogeneous Dirichlet problem
\begin{align}
\displaystyle
- \varDelta u = f \quad \text{in }\Omega, \quad u=0 \quad \text{on }\partial\Omega. \label{eq:poisson-problem}
\end{align}
For any $f \in L^2(\Omega)$, the weak formulation is to find $u \in H^1_0(\Omega)$ such that
\begin{align}
\displaystyle
a(u,v)=(f,v) \quad \forall v \in H^1_0(\Omega), \label{eq:weak-poisson}
\end{align}
where
\begin{align*}
\displaystyle
a(u,v) := \int_\Omega \nabla u \cdot \nabla v dx, \quad (f,v) := \int_\Omega fv dx.
\end{align*}
The Lax--Milgram lemma gives a unique solution $u\in H_0^1(\Omega)$ for each $f\in L^2(\Omega)$, with
\begin{align*}
\displaystyle
| u |_{H^1(\Omega)} \leq C_P(\Omega) \| f \|,
\end{align*}
where $C_P(\Omega)$ is the Poincar\'e constant of $\Omega$. If $\Omega$ is convex, then $u\in H^2(\Omega)$ and 
\begin{align}
\displaystyle
| u |_{H^2(\Omega)} \leq \| \varDelta u \|. \label{cr13}
\end{align}
See, for example, \cite[Theorems 3.1.1.2 and 3.2.1.2]{Gri11}. 

\subsection{Exact-curved finite element approximation}
\label{subsec:discrete-problem}
Fix an integer $k\geq1$, and let $V^k_{h0}\subset H^1_0(\Omega)$ be the exact-curved conforming Lagrange finite element space of degree $k$ defined in \eqref{eq:global-lagrange-space-zero}. The finite element approximation is to find $u_h\in V^k_{h0}$ such that
\begin{align}
\displaystyle
a(u_h,v_h)=(f,v_h) \quad \forall v_h \in V^k_{h0}. \label{eq:discrete-poisson}
\end{align}

\subsection{Error estimates}
\label{subsec:errors}
Throughout this subsection, we suppose that the family $\{\mathbb K_h\}_h$ and its affine cores satisfy the hypotheses of Theorem~\ref{thm:high-order-curved}. For $u$ with $u|_K\in H^{k+1}(K)$ for any $K\in\mathbb K_h$, we define the directional quantity
\begin{align}
\displaystyle
\mathcal A_{h,k}(u)
:=\left[\sum_{K\in\mathbb K_h}
\left(\sum_{|\alpha|=k}h_K^\alpha
|\mathcal D_K^\alpha u|_{H^1(K)}\right)^2\right]^{1/2}.
\label{eq:directional-quantity}
\end{align}
For $k=1$,
\begin{align*}
\displaystyle
\mathcal A_{h,1}(u)
=\left[ \sum_{K \in \mathbb K_h} \left( \sum_{i=1}^{2} h_{K,i} \left| \frac{\partial u}{\partial \tau_{K,i}} \right|_{H^1(K)} \right)^2 \right]^{1/2}.
\end{align*}

\begin{theorem}[Energy-norm error estimate]
\label{thm:poisson-H1-error}
Let $u \in H^1_0(\Omega)$ be the solution of \eqref{eq:weak-poisson}, and let $u_h\in V^k_{h0}$ be the solution of \eqref{eq:discrete-poisson}. If $u|_K \in H^{k+1}(K)$ for any $K \in \mathbb K_h$, then
\begin{align}
\displaystyle
|u-u_h|_{H^1(\Omega)}
\leq C_{H^1}^{E} \mathcal A_{h,k}(u), \label{eq:poisson-H1-anisotropic-error}
\end{align}
where the constant $C_{H^1}^{E}$ is independent of $h$, $K$, and the anisotropic shape of the affine core $T_K$, but depends on $k$, on the fixed bound $\gamma_0$, and on the uniform bounds for $\Psi_K$ and $\Psi_K^{-1}$.
\end{theorem}

\begin{proof}
Galerkin orthogonality gives
\begin{align*}
\displaystyle
a(u-u_h,v_h)=0 \quad \forall v_h \in V^k_{h0}.
\end{align*}
Therefore, C\'ea's lemma gives
\begin{align*}
\displaystyle
|u-u_h|_{H^1(\Omega)} \leq \inf_{v_h \in V^k_{h0}} |u-v_h|_{H^1(\Omega)}.
\end{align*}
Because $u\in H_0^1(\Omega)$ and $u|_K\in H^{k+1}(K)$ with $k+1\geq2$ for any $K$, the continuous representatives of the element restrictions have matching traces on common edges and vanish on boundary edges. Thus, Lemma~\ref{lem:global-interpolation-compatibility} applies and $I_h^ku\in V_{h0}^k$. Taking $v_h := I_h^ku$ and applying \eqref{eq:high-order-curved-H1} with $s=k+1$ on each element, we obtain
\begin{align*}
\displaystyle
|u-u_h|_{H^1(\Omega)}^2
&\leq |u-I_h^ku|_{H^1(\Omega)}^2
= \sum_{K \in \mathbb K_h} |u-I_{L,K}^ku|_{H^1(K)}^2 \\
&\leq C_{k,k+1}^2\gamma_0^2 \sum_{K \in \mathbb K_h} \left( \sum_{|\alpha|=k}h_K^\alpha |\mathcal D_K^\alpha u|_{H^1(K)} \right)^2,
\end{align*}
which is \eqref{eq:poisson-H1-anisotropic-error}.
\end{proof}

\begin{corollary}[Standard $H^1$-error estimate]
\label{cor:standard-H1-error}
Under the assumptions of Theorem~\ref{thm:poisson-H1-error}, suppose in addition that $u\in H^{k+1}(\Omega)$. Then,
\begin{align}
\displaystyle
 |u-u_h|_{H^1(\Omega)}
 \leq
 C_{H^1}^{S} h^k \|u\|_{H^{k+1}(\Omega)}, \label{eq:standard-H1-error}
\end{align}
where the constant $C_{H^1}^{S}$ is independent of $h$, $K$, and the anisotropic shape of the affine core $T_K$, but depends on $k$, $\gamma_0$, and the uniform bounds for $\Psi_K$ and $\Psi_K^{-1}$.
\end{corollary}

\begin{proof}
C\'ea's lemma, Lemma~\ref{lem:global-interpolation-compatibility}, and \eqref{eq:high-order-standard-H1} with $s=k+1$ give
\begin{align*}
|u-u_h|_{H^1(\Omega)}^2
&\leq |u-I_h^ku|_{H^1(\Omega)}^2 \\
&\leq C_{k,k+1}^2 h^{2k}
\sum_{K\in\mathbb K_h}
\|u\|_{H^{k+1}(K)}^2
= C_{k,k+1}^2 h^{2k}\|u\|_{H^{k+1}(\Omega)}^2.
\end{align*}
\end{proof}

\begin{theorem}[$L^2$-error estimate]
\label{thm:poisson-L2-error}
Under the assumptions of Theorem~\ref{thm:poisson-H1-error}, suppose in addition that $\Omega$ is convex. Then,
\begin{align}
\displaystyle
 \|u-u_h\|_{L^2(\Omega)}
\leq C_{L^2}^{E} h \mathcal A_{h,k}(u),
\label{eq:poisson-L2-anisotropic-error}
\end{align}
where the constant $C_{L^2}^{E}$ is independent of $h$, $K$, and the anisotropic shape of the affine core $T_K$, but depends on $k$, $\gamma_0$, the uniform bounds for $\Psi_K$ and $\Psi_K^{-1}$, and the domain $\Omega$.
\end{theorem}

\begin{proof}
We set $e_h := u - u_h$ and assume that $e_h\neq 0$. Let  $z \in H^2(\Omega) \cap H_0^1(\Omega)$ satisfy
\begin{align}
\displaystyle
a(\varphi , z ) = ( \varphi , e_h) \quad \forall \varphi \in H_0^1(\Omega). \label{Lag4=9}
\end{align}
By \eqref{cr13}, $|z|_{H^2(\Omega)} \leq \| e_h \|_{L^2(\Omega)}$. Furthermore, taking $\varphi=z$ in \eqref{Lag4=9} and using Poincar\'e's inequality yields
\begin{align*}
|z|_{H^1(\Omega)}^2
=(e_h,z)
\leq C_P(\Omega)\|e_h\|_{L^2(\Omega)}|z|_{H^1(\Omega)},
\end{align*}
and hence $|z|_{H^1(\Omega)}\leq C_P(\Omega)\|e_h\|_{L^2(\Omega)}$.
By Lemma~\ref{lem:global-interpolation-compatibility}, $I_h^kz\in V_{h0}^k$. The Galerkin orthogonality then gives
\begin{align*}
\displaystyle
\| e_h \|_{L^2(\Omega)}^2
= a(e_h,z) = a(e_h,z - I_h^k z)
\leq |e_h|_{H^1(\Omega)} |z - I_h^k z|_{H^1(\Omega)}.
\end{align*}
To display the contribution of the non-affine map, set
\begin{equation}
\begin{aligned}
M_1&:=\sup_{K\in\mathbb K_h}\|D\Psi_K\|_{L^\infty(T_K)},\\
M_{-1}&:=\sup_{K\in\mathbb K_h}\|D\Psi_K^{-1}\|_{L^\infty(K)},\\
M_2&:=\sup_{K\in\mathbb K_h}\|D^2\Psi_K\|_{L^\infty(T_K)}.
\end{aligned}
\label{eq:map-bending-constants}
\end{equation}
These quantities are bounded uniformly over the mesh family by Assumption~\ref{ass:curved-correction}.  By \eqref{eq:transported-direction} and \eqref{eq:transported-direction-derivative}, $|\tau_{K,i}|\leq M_1$ and $|D\tau_{K,i}|\leq M_2M_{-1}$ on $K$. Because $\nabla(\tau_{K,i}\cdot\nabla z)=(D\tau_{K,i})^{\top}\nabla z+D^2z\,\tau_{K,i}$ almost everywhere,
\begin{align*}
|\mathcal D_{K,i} z|_{H^1(K)}
\leq M_1|z|_{H^2(K)}+M_2M_{-1}|z|_{H^1(K)}.
\end{align*}
Applying \eqref{eq:high-order-curved-H1} with $s=2$, which is available for any $k\geq1$, using $H_{T_K}/h_{T_K}\leq\gamma_0$ and $h_{K,i}\leq h$, and summing over $K$ therefore gives
\begin{align}
|z-I_h^kz|_{H^1(\Omega)}
\leq C_0h\bigl(M_1|z|_{H^2(\Omega)}
             +M_2M_{-1}|z|_{H^1(\Omega)}\bigr).
\label{eq:dual-interpolation-curvature}
\end{align}
Here, $C_0$ depends on $k$, the fixed affine-core shape bound, and the first-order map-transfer bounds.  The second term comes from $D^2\Psi_K$. Together with the bounds for $|z|_{H^2(\Omega)}$ and $|z|_{H^1(\Omega)}$, \eqref{eq:dual-interpolation-curvature} gives
\begin{align}
|z-I_h^kz|_{H^1(\Omega)}
\leq C_0h\bigl(M_1+M_2M_{-1}C_P(\Omega)\bigr)
       \|e_h\|_{L^2(\Omega)}
\leq C h\|e_h\|_{L^2(\Omega)},
\label{eq:dual-global-interpolation}
\end{align}
where $C$ depends on $k$, $\gamma_0$, the uniform curved-map bounds, and $\Omega$. Because the elements cover $\Omega$ exactly, no separate domain approximation error arises; the non-affine contribution to the interpolation error remains.  We conclude that
\begin{align*}
\|u-u_h\|_{L^2(\Omega)}
\leq C h |u-u_h|_{H^1(\Omega)},
\end{align*}
and \eqref{eq:poisson-H1-anisotropic-error} proves the assertion.
\end{proof}

\begin{corollary}[Standard $L^2$-error estimate]
\label{cor:standard-L2-error}
Under the assumptions of Theorem~\ref{thm:poisson-L2-error}, suppose that $u\in H^{k+1}(\Omega)$. Then,
\begin{align}
\|u-u_h\|_{L^2(\Omega)}
\leq C_{L^2}^{S} h^{k+1} \|u\|_{H^{k+1}(\Omega)},
\label{eq:standard-L2-error}
\end{align}
where the constant $C_{L^2}^{S}$ is independent of $h$ and of the anisotropic aspect ratios of the affine cores, but depends on $k$, $\gamma_0$, the uniform bounds for $\Psi_K$ and $\Psi_K^{-1}$, and the domain $\Omega$.
\end{corollary}

\begin{proof}
The proof of Theorem~\ref{thm:poisson-L2-error} gives $\|u-u_h\|_{L^2(\Omega)}\leq C h |u-u_h|_{H^1(\Omega)}$. Applying Corollary~\ref{cor:standard-H1-error} yields \eqref{eq:standard-L2-error}.
\end{proof}

\begin{remark}
For $k=1$, the argument leading to \eqref{eq:dual-interpolation-curvature}, applied to $u$ instead of $z$, gives the primal bound
\begin{align*}
|u-u_h|_{H^1(\Omega)}
\leq C h\bigl(M_1|u|_{H^2(\Omega)}
+M_2M_{-1}|u|_{H^1(\Omega)}\bigr).
\end{align*}
Thus, $M_2$ appears in both the primal and the dual interpolation errors; since it is bounded uniformly, the order $h^2$ is retained.  In general, the orders $h^k$ and $h^{k+1}$ require the $H^{k+1}$ regularity of $u$.  The convexity of $\Omega$ gives the $H^2$ regularity of the dual solution, but it does not by itself give $H^{k+1}$ regularity of $u$ for a general $f\in L^2(\Omega)$.  The estimates concern exact Galerkin integrals; numerical quadrature introduces an additional consistency error.
\end{remark}

\section{Assembly on exact curved elements}
\label{sec:numerical-algorithm}
Let $\{\hat\varphi_i\}_{i=1}^{N_k}$ be the degree-$k$ basis on $\widehat T$, where $N_k=(k+1)(k+2)/2$, and we set $\varphi_{K,i}=\hat\varphi_i\circ F_K^{-1}$.  With $J_K(\hat x)=D\Psi_K(\Phi_K^{\mathrm{aff}}(\hat x)) D\Phi_K^{\mathrm{aff}}$, the pullback gives
\begin{align*}
\nabla_x\varphi_{K,i}(F_K(\hat x))
=J_K(\hat x)^{-\top}\nabla_{\hat x}\hat\varphi_i(\hat x).
\end{align*}
The element stiffness and load entries are
\begin{align*}
A^K_{ij}
&=\int_{\widehat T}
\bigl(J_K^{-\top}\nabla\hat\varphi_j\bigr)
\cdot\bigl(J_K^{-\top}\nabla\hat\varphi_i\bigr)
|\det J_K|\,d\hat x,\\
b^K_i
&=\int_{\widehat T}f(F_K(\hat x))\hat\varphi_i(\hat x)
|\det J_K(\hat x)|\,d\hat x.
\end{align*}
Because adjacent elements share the edge parametrisations, the edge nodes are numbered globally as in the affine case.  In the experiments below, $\Psi_K$ and $D\Psi_K$ are evaluated analytically at each quadrature point, so the circular boundary and the element maps are not replaced by polynomial approximations.  The estimates in Section~\ref{model=sec} concern exact Galerkin integrals. We check the quadrature error numerically in Section~\ref{sec:numerical-experiment}.

\section{Numerical experiments on exact curved anisotropic meshes}
\label{sec:numerical-experiment}
We test the error estimates on the explicit family of Section~\ref{subsec:explicit-exact-family} and, in Section~\ref{subsec:fixed-h}, on a variant with a fixed horizontal mesh width. Every element uses the analytic map $\Psi(s,t)=(s,tg(s))$ of \eqref{eq:cap-domain-map}, so the upper boundary is an exact circular arc at every refinement level.  The affine cores have side lengths $\delta=N^{-1}$ and $\delta^2$, so their aspect ratio is $N$ and $H_T/h_T=2$ for any element.  Figure~\ref{fig:exact-cap-mesh} shows the first mesh used in the experiments; Table~\ref{tab:semi-regular-geometry-check} verifies the geometric parameter directly from the affine mesh coordinates.

\subsection{Numerical verification of the affine semi-regularity}
\label{subsec:numerical-semi-regularity}
The semi-regularity parameter is defined on the affine cores, not on the curved images. We computed this parameter independently of the Poisson solver, using the vertex coordinates of every triangle generated by the mesh code. For each affine triangle, let $\ell_1\leq\ell_2\leq\ell_3$ denote its three edge lengths and let $|T|_2$ be its area. Because $\ell_3=h_T$ and $\ell_2=h_{K,1}$, $\ell_1=h_{K,2}$ in the present right-triangle families, the quantity evaluated is
\begin{align}
\frac{H_T}{h_T}=\frac{\ell_1\ell_2}{|T|_2}.
\label{eq:numerically-evaluated-semi-regularity}
\end{align}
Table~\ref{tab:semi-regular-geometry-check} shows the minimum and maximum values obtained over all the elements, including the two triangles with reversed diagonals at the corners. For every tested mesh, both extrema equal $2$ to the displayed precision. This agrees with the exact identity \eqref{eq:cap-semi-regularity}; in particular, the affine semi-regularity constant $\gamma_0=2$ is independent of the increasing aspect ratio. The curved correction is controlled separately by the bounds on $\Psi$ and $\Psi^{-1}$.

\begin{table}[htbp]
\centering
\caption{Direct geometric verification on the affine cores. The extrema are taken over all triangles in each mesh. The fixed-width family has $N=8$, and $m$ denotes the affine-core aspect ratio.}
\label{tab:semi-regular-geometry-check}
\begin{tabular}{llrrr}
\toprule
Family & Parameter & Triangles & $\min_T H_T/h_T$ & $\max_T H_T/h_T$ \\
\midrule
Refinement & $N=4$  & 128   & 2.000000 & 2.000000 \\
Refinement & $N=6$  & 432   & 2.000000 & 2.000000 \\
Refinement & $N=8$  & 1024  & 2.000000 & 2.000000 \\
Refinement & $N=12$ & 3456  & 2.000000 & 2.000000 \\
Refinement & $N=16$ & 8192  & 2.000000 & 2.000000 \\
\midrule
Fixed width & $m=1$   & 128   & 2.000000 & 2.000000 \\
Fixed width & $m=4$   & 512   & 2.000000 & 2.000000 \\
Fixed width & $m=16$  & 2048  & 2.000000 & 2.000000 \\
Fixed width & $m=64$  & 8192  & 2.000000 & 2.000000 \\
Fixed width & $m=256$ & 32768 & 2.000000 & 2.000000 \\
\bottomrule
\end{tabular}
\end{table}

\subsection{Manufactured problem and implementation}
We set $a=1/2$ and choose
\begin{align}
u(x,y)&=(a^2-x^2)y\bigl(g(x)-y\bigr),
\label{eq:cap-exact-solution}\\
f(x,y)&=-\Delta u(x,y)\notag\\
&=2\bigl(g(x)y-y^2\bigr)+4xg'(x)y  -(a^2-x^2)g''(x)y+2(a^2-x^2).
\label{eq:cap-forcing}
\end{align}
To test convergence for a non-polynomial pullback, rather than only for $u\circ\Psi=(a^2-s^2)(1-s^2)t(1-t)$, we also use
\begin{align}
U(x,y)&=e^{x+y}u(x,y),\notag\\
F(x,y)&=-\Delta U(x,y)
=e^{x+y}\bigl(f-2u_x-2u_y-2u\bigr).
\label{eq:cap-exponential-solution}
\end{align}
Both solutions vanish on the bottom, the two straight sides, and the exact circular top boundary, and both are smooth on $\overline\Omega$.

We use the mapped $P_1$ and $P_2$ Lagrange spaces.  Element integrals are approximated by a Duffy transformation followed by an $8\times8$ tensor Gauss rule, and the error integrals by a separate $12\times12$ rule.  Homogeneous Dirichlet data are imposed at all boundary nodes.  We use diagonally preconditioned conjugate gradients with a target of $10^{-11}$ for the recursively updated relative residual. At termination, the true residual $b-Ax_h$ is recomputed. We evaluate an upper bound for the normwise backward error using $\|A\|_{\infty}\geq\max_i|A_{ii}|$. When the true relative residual exceeds the target because of floating-point effects, the solver accepts the iterate only if this backward-error upper bound is at the level of round-off. The two stopping criteria are recorded separately.

For the first solution, we report the relative errors
\begin{align*}
E_{H^1}:=\frac{|u-u_h|_{H^1(\Omega)}}{|u|_{H^1(\Omega)}},
\quad
E_{L^2}:=\frac{\|u-u_h\|_{L^2(\Omega)}}{\|u\|_{L^2(\Omega)}},
\end{align*}
where $u_h\in V_{h0}^k$ is the solution of \eqref{eq:discrete-poisson} and the norms of $u$ are evaluated with the same $12\times12$ rule. The observed rate between consecutive values of $N$ is
\begin{align*}
\operatorname{rate}
=\frac{\log(E_{\rm old}/E_{\rm new})}
       {\log(N_{\rm new}/N_{\rm old})}.
\end{align*}
For the second solution, the same symbols denote the relative errors for $U$. Because the global affine-core mesh size is $h=\delta\sqrt{1+\delta^2}$, rates measured with $\delta$ and with $h$ have the same limiting values.

\subsection{Results}
\label{subsec:exact-results}
Table~\ref{tab:exact-anisotropic-errors} shows convergence on meshes whose aspect ratios increase without bound.  The measured rates approach $k$ in the $H^1$-seminorm and $k+1$ in the $L^2$-norm for $k=1,2$, as predicted by Corollaries~\ref{cor:standard-H1-error} and~\ref{cor:standard-L2-error}.

\begin{table}[htbp]
\centering
\caption{Relative errors for $u$ on the exact circular-arc domain. Each mesh has $2N^3$ triangles, affine-core aspect ratio $N$, and $H_T/h_T=2$. Rates use $\delta=N^{-1}$.}
\label{tab:exact-anisotropic-errors}
\begin{tabular}{rrrrrrr}
\toprule
$k$ & $N$ & triangles & $E_{H^1}$ & rate & $E_{L^2}$ & rate \\
\midrule
1 & 4  & 128   & $2.878990\times10^{-1}$ & --    & $8.705648\times10^{-2}$ & -- \\
1 & 6  & 432   & $1.932020\times10^{-1}$ & 0.984 & $3.890057\times10^{-2}$ & 1.987 \\
1 & 8  & 1024  & $1.449680\times10^{-1}$ & 0.998 & $2.186283\times10^{-2}$ & 2.003 \\
1 & 12 & 3456  & $9.661974\times10^{-2}$ & 1.001 & $9.701117\times10^{-3}$ & 2.004 \\
1 & 16 & 8192  & $7.244976\times10^{-2}$ & 1.001 & $5.452571\times10^{-3}$ & 2.003 \\
\midrule
2 & 4  & 128   & $3.028151\times10^{-2}$ & --    & $3.541600\times10^{-3}$ & -- \\
2 & 6  & 432   & $1.292711\times10^{-2}$ & 2.099 & $1.028404\times10^{-3}$ & 3.050 \\
2 & 8  & 1024  & $7.175771\times10^{-3}$ & 2.046 & $4.310864\times10^{-4}$ & 3.022 \\
2 & 12 & 3456  & $3.159429\times10^{-3}$ & 2.023 & $1.271328\times10^{-4}$ & 3.012 \\
2 & 16 & 8192  & $1.771318\times10^{-3}$ & 2.011 & $5.354358\times10^{-5}$ & 3.006 \\
\bottomrule
\end{tabular}
\end{table}

The second solution $U=e^{x+y}u$ has a non-polynomial pullback by $\Psi$.  Table~\ref{tab:exact-anisotropic-exponential} gives its errors on the same exact curved meshes.  The rates between $N=8$ and $N=16$ are close to one and two for $P_1$, and two and three for $P_2$, in the $H^1$- and $L^2$-norms respectively.

\begin{table}[htbp]
\centering
\caption{Relative errors for $U=e^{x+y}u$ on the exact circular-arc meshes.  Rates use $\delta=N^{-1}$.}
\label{tab:exact-anisotropic-exponential}
\begin{tabular}{rrrrrr}
\toprule
$k$ & $N$ & $E_{H^1}$ & rate & $E_{L^2}$ & rate \\
\midrule
1 & 8  & $1.580044\times10^{-1}$ & --    & $2.484752\times10^{-2}$ & -- \\
1 & 16 & $7.908565\times10^{-2}$ & 0.998 & $6.245593\times10^{-3}$ & 1.992 \\
\midrule
2 & 8  & $8.868662\times10^{-3}$ & --    & $5.195619\times10^{-4}$ & -- \\
2 & 16 & $2.163014\times10^{-3}$ & 2.036 & $6.451169\times10^{-5}$ & 3.010 \\
\bottomrule
\end{tabular}
\end{table}

To assess the quadrature error, we repeated the $N=8,16$ computations for both solutions with a $12\times12$ assembly rule and a $16\times16$ error rule.  The relative changes were below $2\times10^{-14}$ in $E_{H^1}$ and below $10^{-9}$ in $E_{L^2}$.  The changes did not decrease systematically under further refinement of the rules and are consistent with numerical round-off. Tightening the linear-solver tolerance from $10^{-11}$ to $10^{-13}$ at $N=16$, $k=2$ likewise left the displayed errors unchanged.

\subsection{Aspect ratio at a fixed horizontal mesh width}
\label{subsec:fixed-h}
In Table~\ref{tab:exact-anisotropic-errors}, the aspect ratio $N$ and the mesh size $h\approx N^{-1}$ change together.  To separate the two, we fix $N=8$, so that $\delta=1/8$, and divide $[0,1]$ into $8m$ intervals instead of $N^2$, with $m\in\{1,4,16,64,256\}$.  The rectangles are split as in Section~\ref{subsec:explicit-exact-family}, including the two corner-diagonal flips.  The affine cores are right triangles with legs $\delta$ and $\delta/m$; hence their aspect ratio is $m$, $H_T/h_T=2$, and the global mesh size $h=\delta(1+m^{-2})^{1/2}$ lies between $\delta$ and $\sqrt2\,\delta$. By Proposition~\ref{prop:global-map}, the map bounds are those of the family in Table~\ref{tab:exact-anisotropic-errors}.  The case $m=8$ coincides with the mesh $N=8$ there.

Table~\ref{tab:fixed-h} gives the relative errors for $u$.  For both degrees and both norms, the errors decrease monotonically over the tested range and appear to level off; no growth with the aspect ratio is observed up to $m=256$.  This is consistent with Theorem~\ref{thm:high-order-curved}.  For $m>1$, the leg of length $\delta$ is parallel to the $s$-axis in any core, so that $h_{K,1}=\delta$ and $r_1=\pm(1,0)^\top$ for all elements.  As $h_{K,2}=\delta/m\to0$, only the terms with $\alpha=(s,0)$ and $\alpha=(s-1,0)$ remain in \eqref{eq:high-order-curved-L2} and \eqref{eq:high-order-curved-H1}, and their sum over the mesh does not depend on $m$.

\begin{table}[htbp]
\centering
\caption{Relative errors for $u$ at the fixed $\delta=1/8$, with $8m$
intervals in $t$.  The affine cores have aspect ratio $m$ and
$H_T/h_T=2$, and the mesh has $128m$ triangles.}
\label{tab:fixed-h}
\begin{tabular}{rrrrr}
\toprule
$k$ & $m$ & triangles & $E_{H^1}$ & $E_{L^2}$ \\
\midrule
1 & 1   & 128   & $1.967658\times10^{-1}$ & $4.126493\times10^{-2}$ \\
1 & 4   & 512   & $1.478199\times10^{-1}$ & $2.285645\times10^{-2}$ \\
1 & 16  & 2048  & $1.442469\times10^{-1}$ & $2.161049\times10^{-2}$ \\
1 & 64  & 8192  & $1.440208\times10^{-1}$ & $2.153099\times10^{-2}$ \\
1 & 256 & 32768 & $1.440067\times10^{-1}$ & $2.152600\times10^{-2}$ \\
\midrule
2 & 1   & 128   & $1.511042\times10^{-2}$ & $1.067376\times10^{-3}$ \\
2 & 4   & 512   & $7.665184\times10^{-3}$ & $4.596514\times10^{-4}$ \\
2 & 16  & 2048  & $7.049301\times10^{-3}$ & $4.240220\times10^{-4}$ \\
2 & 64  & 8192  & $7.009387\times10^{-3}$ & $4.218216\times10^{-4}$ \\
2 & 256 & 32768 & $7.006885\times10^{-3}$ & $4.216841\times10^{-4}$ \\
\bottomrule
\end{tabular}
\end{table}

On the meshes where the quadrature and solver checks were repeated, the displayed values in Table~\ref{tab:fixed-h} changed by less than one unit in the final digit.  The number of conjugate-gradient iterations grows with $m$; preconditioners suited to large aspect ratios are not considered here. In all reported computations, the true relative residual is at most $1.2\times10^{-8}$, attained for $m=256$ and $k=2$. Whenever it exceeds $10^{-11}$, the backward-error upper bound is below $2\times10^{-14}$. The reported finite-element errors are unchanged at the displayed precision when compared with the original solver, and selected strongly anisotropic cases were also checked with a sparse direct solver.

\section{Concluding remarks}
\label{sec:concluding-remarks}
We have combined anisotropic interpolation estimates on the affine cores with the exact curved element maps $F_K=\Psi_K\circ\Phi_K^{\mathrm{aff}}$ for any fixed Lagrange degree $k\ge1$.  The compatibility of the edge parametrisations gives a conforming space, and the uniform bounds for the curved maps allow the affine estimates to be transferred with constants independent of the aspect ratios.  The resulting estimates for the Poisson problem have the usual orders $h^k$ in the energy norm and $h^{k+1}$ in the $L^2$ norm for solutions in $H^{k+1}$; since the element maps are non-affine, the interpolation bounds contain lower-order terms.

Proposition~\ref{prop:global-map} shows that any semi-regular triangulation of a polygon, mapped by a fixed smooth diffeomorphism, satisfies these geometric assumptions, whatever its aspect ratios.  The explicit circular-arc family is an example with $H_T/h_T=2$ while the aspect ratio tends to infinity.  The $P_1$ and $P_2$ computations on this exact geometry show the predicted orders. At a fixed horizontal mesh width, the errors remain bounded as the aspect ratio increases to $256$. Constructing elementwise curved corrections with uniform bounds on more general smooth domains, and extending the analysis to three dimensions, remain subjects for future work.

\paragraph{Funding.}
No funding was received for this study.

\paragraph{Data availability.}
The numerical values reported in Section~\ref{sec:numerical-experiment} and the machine-readable results are provided as ancillary files with this arXiv submission.

\paragraph{Code availability.}
The Python source code and instructions for reproducing the numerical results in Section~\ref{sec:numerical-experiment} are provided as ancillary files with this arXiv submission.

\paragraph{Competing interests.}
The author declares no competing interests.

\paragraph{Use of AI tools.}
The research questions, mathematical results, and conclusions are the author's own. The author used ChatGPT (OpenAI) for language editing and assistance with checking selected mathematical derivations, numerical code, and manuscript presentation. The author reviewed the AI-assisted material and takes full responsibility for the manuscript.


\end{document}